\documentclass[lettersize,onecolumn]{IEEEtran}
\usepackage{amsmath,amsfonts,amsthm,amssymb}
\usepackage{algorithm}
\usepackage{hyperref}
\usepackage{algpseudocode}
\usepackage{array}
\usepackage[caption=false,font=normalsize,labelfont=sf,textfont=sf]{subfig}
\usepackage{textcomp}
\usepackage{url}
\usepackage{verbatim}
\usepackage{graphicx}
\usepackage{cite}
\usepackage{wrapfig}
\usepackage{dsfont}
\usepackage{tikz}
\usepackage{booktabs}
\usetikzlibrary{arrows.meta,calc,positioning}
\newtheorem{theorem}{Theorem}[section]
\newtheorem{proposition}[theorem]{Proposition}
\newtheorem{lemma}[theorem]{Lemma}
\newtheorem{definition}[theorem]{Definition}
\newtheorem{remark}[theorem]{Remark}

\newif\ifCOM
\COMfalse

\newcommand\eqdef{\overset{\mbox{\tiny def}}{=}}
\newcommand{\HH}{\mathbb H}
\newcommand{\FF}{\mathbb F}
\newcommand{\ZZ}{\mathbb Z}
\newcommand{\CC}{\mathcal C}
\newcommand{\DD}{\mathcal D}
\newcommand{\BB}{\mathcal B}
\newcommand{\EE}{\mathcal E}
\newcommand{\MM}{\mathcal M}
\newcommand{\XX}{\mathcal X}
\newcommand{\Tau}{\mathcal T}
\newcommand{\Lift}{\mathcal L}
\newcommand{\ol}[1]{\overline{#1}}
\newcommand{\RePart}{\operatorname{Re}}
\newcommand{\ip}[2]{\left\langle #1,#2\right\rangle}
\newcommand{\Pow}[1]{2^{#1}}
\newcommand{\RR}{\ensuremath{\mathbb{R}}}
\newcommand{\QQ}{\ensuremath{\mathcal{Q}}_8}

\newcommand{\Ecal}{\mathcal{E}}
\newcommand{\Rcal}{\mathcal{R}}
\newcommand{\Xcal}{\mathcal{X}}

\begin{document}

\title{A quaternionic construction behind \\$841$-point kissing arrangement in ${\mathbb R}^{12}$}

\author{Rustem~Takhanov
\thanks{R.~Takhanov is with the Mathematics Department, Nazarbayev University, and Nazarbayev University Research Administration, Astana, Kazakhstan, e-mail: rustem.takhanov@nu.edu.kz.}
}



\maketitle
\begin{abstract}
Recently, a new record kissing arrangement of $841$ points in $\mathbb R^{12}$ was obtained numerically by optimization (Takhanov--Assylbekov--Yun, 2026). The configuration was released as a coordinate file, without a mathematical description of its structure. The purpose of this paper is to provide such a description. The key observation is that the geometry becomes transparent once we regard
$\mathbb R^{12}\cong \mathbb H^3$
as the Cartesian product of three copies of the quaternion algebra.

We first introduce a new $840$-point kissing arrangement with a certain quaternionic structure. It consists of three mutually orthogonal regular $24$-cells, supported on the three quaternionic coordinate factors
$\mathbb H\times\{0\}\times\{0\}$, 
$\{0\}\times\mathbb H\times\{0\}$,
$\{0\}\times\{0\}\times\mathbb H$,
together with two $384$-point families obtained by lifting affine sets of the form
\[
\{(u,v,w)\in (\mathbb F_2^2)^3\mid u+v+w=\eta\},
\]
to quaternionic triples (whose components belong to the binary octahedral group $2O$) and then applying suitable component-wise rotations and weightings.
 A characteristic feature of this construction is a pronounced asymmetry among the three quaternionic factors. For the $816$ vectors obtained after removing the third $24$-cell, most of the squared norm is concentrated in the first two quaternionic coordinates, while the third coordinate carries systematically less mass. Thus, the third four-dimensional factor contains more available space than the first two.

We then show that this $840$-point configuration provides a natural structural model for the numerical $841$-point record. After an appropriate orthogonal alignment and matching of the points, $816$ vectors of the record configuration are seen to be only slightly deformed from the corresponding $816$ vectors of our $840$-point arrangement. By contrast, the $24$ vectors belonging to the third $24$-cell undergo a substantially larger deformation, and one additional distinguished vector appears predominantly in the third quaternionic factor. In this way, the $841$-point arrangement can be understood as a controlled deformation of a highly structured $840$-point quaternionic arrangement: the first $816$ points remain nearly fixed, while the comparatively ``free'' third quaternionic factor reorganizes to accommodate the $841$st point.

Finally, we introduce a notion of the general quaternionic construction in dimensions divisible by $4$, and check that record kissing arrangements in ${\mathbb R}^{4k}$, $k\leq 5$, admit a quaternionic construction. 
\end{abstract}

\begin{IEEEkeywords}
Kissing number, quaternions, binary octahedral group, $24$-cell, deformation of quaternionic construction.
\end{IEEEkeywords}

\section{Introduction}

For a positive integer $d$, the \emph{kissing number} $\tau_d$ is the largest number of nonoverlapping unit balls in $\mathbb{R}^d$ that can simultaneously touch a fixed unit ball. Equivalently, $\tau_d$ is the maximum cardinality of a set $X\subset S^{d-1}$ such that
\[
    \langle x,y\rangle\leq \frac12
    \qquad\text{for all distinct }x,y\in X.
\]
Such a set will be called a \emph{kissing arrangement}. Although the definition is elementary, exact kissing numbers are known in only a few dimensions~\cite{Schütte1952,2003RuMaS,zbMATH03680654,ODLYZKO1979210}, and the problem remains one of the basic sources of difficult and beautiful questions in discrete geometry.

Consequently, lower bounds for $\tau_d$ are of substantial interest. In the dimensions up to $d=9$, the best known lower bounds appear to be strong candidates for the true values. Currently, the main interest here is the intermediate range
$10\leq d\leq 20$. Recent examples include new constructions in dimensions $10$ and $14$ due to Ganzhinov~\cite{GANZHINOV202512}, and the improvements of Cohn and Li in dimensions $17$ through $21$~\cite{cohn2026improvedkissingnumbersseventeen}. The constructions themselves are equally important as numerical values of lower bounds. A transparent construction reveals which geometric and combinatorial mechanisms create room on the sphere, and these mechanisms can then be combined to produce further examples.

This perspective has become particularly relevant with the recent use of artificial intelligence in mathematical search. Multi-agent search systems based on large language models have already produced new lower bounds for kissing numbers. AlphaEvolve, for example, improved the lower bound in dimension $11$ to $593$~\cite{novikov2025alphaevolvecodingagentscientific}. The EinsteinArena's $604$-point record in ${\mathbb R}^{11}$ is a particularly attractive example. This illustrates how machine-generated data can be transformed into a mathematical construction that can be analyzed, communicated, and reused. An approach based on game-theoretic reinforcement learning allowed to improve lower bounds in dimensions $25$ through $31$~\cite{ma2026findingkissingnumbersgametheoretic}.

The purpose of the present paper is to carry out a structural explanation for the recently discovered $841$-point kissing arrangement in $\mathbb{R}^{12}$~\cite{takhanov2026structurekissingarrangementsmathbb}. The status of this result is, in a sense, ambiguous. As a lower bound it is completely concrete: one may simply provide a coordinate file and verify all pairwise inner products. As a mathematical object, however, it has a large complexity: the file is easy to check but difficult to compress into a short conceptual description. Our main observation is that this complexity is largely superficial. The $841$-point arrangement can be understood as a small deformation of a much more simple $840$-point arrangement with a quaternionic structure.

More precisely, we identify
$\mathbb{R}^{12}\cong \mathbb{H}^3$
and construct a new $840$-point kissing arrangement from three pairwise orthogonal regular $24$-cells, one in each quaternionic coordinate factor, together with two $384$-point families defined by ternary affine relations over ${\mathbb F}_2^2$ and lifted to triples whose entries lie in the binary octahedral group $2O$. Suitable quaternionic rotations and unequal coordinate weights produce the required vectors. This $840$-point configuration is distinct from the $1579$ non-isometric $840$-point arrangements obtained from signed Johnson graphs by Takhanov and Yun~\cite{takhanov2026classificationindependentsetssigned}, and it is also different from the Clebsch-graph-based $840$-point constructions used in Takhanov--Assylbekov--Yun~\cite{takhanov2026structurekissingarrangementsmathbb}. The new quaternionic model explains the numerical $841$-point record remarkably well: after one orthogonal alignment and a relabeling, $816$ of the points are only slightly deformed, while the substantial change is concentrated in the remaining $24$ points of the third axial $24$-cell together with the additional $841$st point. The deformation itself has further low-dimensional structure, but we are not interested in that structure here; our concern is the global quaternionic explanation of the configuration.

The paper is organized as follows. Section~\ref{preliminaries} is dedicated to standard facts from theory of quaternions --- the quaternionic realization of the two regular $24$-cells and the binary octahedral group $2O$. Section~\ref{sec:canonical-relations} introduces the sign lifts and ternary affine relations used in the construction. Section~\ref{sec:canonical-arrangement} gives the quaternionic $840$-point kissing arrangement and states its kissing property, whose detailed verification is made in Section~\ref{main-proof} of Appendix. The fact that this arrangement is not isomorphic to previously known $840$-point arrangements in ${\mathbb R}^{12}$ is verified in Section~\ref{sec:distinction-previous-840} of Appendix. Section~\ref{sec:comparison-canonical-840} compares this canonical arrangement with the numerical $841$-point configuration and exhibits the separation between the nearly unchanged $816$-point core and the exceptional $25$-point part. Section~\ref{quaternionic-general} places the construction in a broader framework of quaternionic kissing arrangements and discusses the $24$-cell, $E_8$ examples. Subsection~\ref{barnes} discusses the Barnes--Wall and Cohn--Li configurations. 

\section{Preliminaries: quaternionic realization of the two regular $24$-cells}\label{preliminaries}
This section is dedicated to introduction of well-known definitions of quaternion algebra and classical facts from theory of $24$-cells.

Let $\HH$ be the algebra of real quaternions, with standard basis
$1,\mathbf i,\mathbf j,\mathbf k$, conjugation $q\mapsto\ol q$, and
Euclidean inner product
\[
    \ip{p}{q}=\RePart(p\ol q).
\]

We identify the Euclidean space $\RR^4$ with $\HH$. Under this identification, the Euclidean norm agrees with the
quaternionic norm $|q|=\sqrt{q\overline{q}}$.
The quaternionic norm is multiplicative, i.e. $|pq|=|p|\cdot |q|$.
Consequently, left and right multiplication by a unit quaternion are
orthogonal transformations of $\RR^4$. We also identify $\HH^3$ with $\mathbb R^{12}$.

We denote the quaternion group by $\QQ$, i.e.
\begin{equation}
\QQ
=
\{+1,-1,+\mathbf{i},-\mathbf{i},+\mathbf{j},-\mathbf{j}, +\mathbf{k},-\mathbf{k}\}.
\end{equation}
Let us introduce the following $24$-element subset of the unit sphere in
$\HH$:
\begin{equation}
\CC
=
\QQ
\cup
\left\{
\frac{
\varepsilon_0
+\varepsilon_1\mathbf{i}
+\varepsilon_2\mathbf{j}
+\varepsilon_3\mathbf{k}
}{2}
\mid
\varepsilon_0,\varepsilon_1,\varepsilon_2,\varepsilon_3
\in\{+1,-1\}
\right\}.
\end{equation}
Thus, $\CC$ consists of the eight elements of $\QQ$ and sixteen
half-integral unit quaternions.

Define a second $24$-element set by
\begin{equation}
\DD
=
\left\{
\frac{\varepsilon_r u_r+\varepsilon_s u_s}{\sqrt{2}}
\mid
0\leq r<s\leq3,
\quad
\varepsilon_r,\varepsilon_s\in\{+1,-1\}
\right\},
\end{equation}
where $(u_0,u_1,u_2,u_3)
=
(1,\mathbf{i},\mathbf{j},\mathbf{k})$.
Every element of $\DD$ has exactly two nonzero coordinates, each equal
to $\pm\frac{1}{\sqrt{2}}$.

The set $\CC$ is commonly called the binary tetrahedral group and is denoted by
$2T$ (see, for example, ~\cite{ConwaySmith}).

We first give a direct verification that $\CC$ is multiplicatively
closed. During that we also calculate cosets of $\CC$ w.r.t. subgroup $\QQ$ which will be instrumental in our construction of kissing arrangements.
 
Put
\begin{equation}
    \omega=\frac{1+\mathbf i+\mathbf j+\mathbf k}{2},
    \qquad
    \tau=\frac{1+\mathbf i}{\sqrt2}.
\end{equation}

\begin{lemma}
The set $\CC$ is a subgroup of the group of unit quaternions.
\end{lemma}

\begin{proof}
Recall that
$\omega
=
\frac{1+\mathbf{i}+\mathbf{j}+\mathbf{k}}{2}$
is a quaternion whose conjugation transform $x\mathbf{i}+y\mathbf{j}+z\mathbf{k}\to \omega(x\mathbf{i}+y\mathbf{j}+z\mathbf{k})\omega^{-1}$ corresponds to a rotation by angle $\frac{2\pi}{3}$ around $\frac{1}{\sqrt{3}}(1,1,1)$. Then, $\omega^2
=
\frac{-1+\mathbf{i}+\mathbf{j}+\mathbf{k}}{2}$.
A direct calculation gives
$\omega^3=-1$. 
By construction,
\[
\omega\mathbf{i}\omega^{-1}=\mathbf{j},
\qquad
\omega\mathbf{j}\omega^{-1}=\mathbf{k},
\qquad
\omega\mathbf{k}\omega^{-1}=\mathbf{i}.
\]
Thus $\omega$ normalizes the quaternion group $\QQ$.
Since $\omega$ normalizes $\QQ$ and
$\omega^3=-1\in\QQ$, 
the union
\[
\QQ\sqcup\omega\QQ\sqcup\omega^2\QQ
\]
is the subgroup generated by $\QQ$ and $\omega$. 
Let us prove that
\begin{equation}
\CC
=
\QQ
\sqcup
\omega\QQ
\sqcup
\omega^2\QQ.
\end{equation}
The first coset, denoted by $\CC_0$, is $\QQ$. We calculate the right coset $\omega\QQ$ directly. First,
$\omega\cdot1
=
\frac{1+\mathbf{i}+\mathbf{j}+\mathbf{k}}{2}$.
Next,
\[
\begin{aligned}
\omega\mathbf{i}
=
\frac{\mathbf{i}+\mathbf{i}^2+\mathbf{j}\mathbf{i}+\mathbf{k}\mathbf{i}}{2}
=
\frac{\mathbf{i}-1-\mathbf{k}+\mathbf{j}}{2}
=
\frac{-1+\mathbf{i}+\mathbf{j}-\mathbf{k}}{2}.
\end{aligned}
\]
Similarly,
\[
\begin{aligned}
\omega\mathbf{j}
=
\frac{\mathbf{j}+\mathbf{i}\mathbf{j}+\mathbf{j}^2+\mathbf{k}\mathbf{j}}{2}
=
\frac{\mathbf{j}+\mathbf{k}-1-\mathbf{i}}{2}=
\frac{-1-\mathbf{i}+\mathbf{j}+\mathbf{k}}{2},
\end{aligned}
\]
and
\[
\begin{aligned}
\omega\mathbf{k}=
\frac{\mathbf{k}+\mathbf{i}\mathbf{k}+\mathbf{j}\mathbf{k}+\mathbf{k}^2}{2}
=
\frac{\mathbf{k}-\mathbf{j}+\mathbf{i}-1}{2}
=
\frac{-1+\mathbf{i}-\mathbf{j}+\mathbf{k}}{2}.
\end{aligned}
\]
Multiplication by the negative elements of $\QQ$ gives the negatives of
these four quaternions. We therefore define
\[
\begin{aligned}
\CC_1
=\omega\QQ
=\Biggl\{\pm&\frac{1+\mathbf{i}+\mathbf{j}+\mathbf{k}}{2},
\ \pm\frac{-1+\mathbf{i}+\mathbf{j}-\mathbf{k}}{2},
\pm\frac{-1-\mathbf{i}+\mathbf{j}+\mathbf{k}}{2},
\ \pm\frac{-1+\mathbf{i}-\mathbf{j}+\mathbf{k}}{2}
\Biggr\}.
\end{aligned}
\]
Equivalently,
\[
\CC_1
=
\left\{
\frac{
\varepsilon_0
+\varepsilon_1\mathbf{i}
+\varepsilon_2\mathbf{j}
+\varepsilon_3\mathbf{k}
}{2}
\mid
\varepsilon_r\in\{+1,-1\},
\quad
\varepsilon_0\varepsilon_1\varepsilon_2\varepsilon_3=1
\right\}.
\]
Analogously,
\[
\CC_2=\omega^2\QQ
=
\left\{
\frac{
\varepsilon_0
+\varepsilon_1\mathbf{i}
+\varepsilon_2\mathbf{j}
+\varepsilon_3\mathbf{k}
}{2}
\mid
\varepsilon_r\in\{+1,-1\},
\quad
\varepsilon_0\varepsilon_1\varepsilon_2\varepsilon_3=-1
\right\}.
\]
Therefore $\CC$ is a
subgroup of the unit quaternions.
\end{proof}

We now express the second $24$-cell as a coset of $\CC$. A union of that coset with $\CC$ gives the binary octahedral group, commonly denoted by $2O$.

\begin{proposition} Recall that $\tau=\frac{1+\mathbf{i}}{\sqrt{2}}$. 
We have
$\DD=\tau \CC=\CC\tau$,
and
$2O=\CC\cup \DD$
is closed under quaternion multiplication. More precisely,
\[
\CC\CC=\CC,
\qquad
\CC\DD=\DD\CC=\DD,
\qquad
\DD\DD=\CC.
\]
In particular, $2O$ is a group of order $48$.
\end{proposition}

\begin{proof}
Note that
$\tau^{-1}
=
\frac{1-\mathbf{i}}{\sqrt{2}}$, and
$\tau^2=\mathbf{i}\in \QQ$.
Conjugation by $\tau$ acts on the standard quaternionic basis as
\[
\tau\mathbf{i}\tau^{-1}=\mathbf{i},
\qquad
\tau\mathbf{j}\tau^{-1}=\mathbf{k},
\qquad
\tau\mathbf{k}\tau^{-1}=-\mathbf{j}.
\]
Therefore, conjugation by $\tau$ fixes the real coordinate and acts on
the three imaginary coordinates by a signed permutation.
It follows immediately that conjugation by $\tau$ preserves
$\QQ$.
It also preserves the set
\[
\left\{
\frac{
\varepsilon_0
+\varepsilon_1\mathbf{i}
+\varepsilon_2\mathbf{j}
+\varepsilon_3\mathbf{k}
}{2}
\mid
\varepsilon_r\in\{+1,-1\}
\right\},
\]
because a signed permutation of the imaginary coordinates maps such a
quaternion to another quaternion of the same form. Hence
$\tau \CC\tau^{-1}=\CC$.
Equivalently,
$\tau \CC=\CC\tau$.

We next prove that
$\tau \CC=\DD$.
For the elements of $\QQ$, the products
\[
\tau(\pm1),
\qquad
\tau(\pm\mathbf{i}),
\qquad
\tau(\pm\mathbf{j}),
\qquad
\tau(\pm\mathbf{k})
\]
give the eight elements of $\DD$, the first four supported on the coordinate pairs
$\{1,{\mathbf i}\}$
 and the second four supported on the coordinate pairs $\{{\mathbf j},{\mathbf k}\}$.

Since $|\QQ|=8$, we obtain
\[
\tau\QQ
=
\left\{
\frac{\varepsilon_0+\varepsilon_1\mathbf{i}}{\sqrt{2}}
\mid
\varepsilon_0,\varepsilon_1\in\{\pm1\}
\right\}
\cup
\left\{
\frac{\varepsilon_2\mathbf{j}+\varepsilon_3\mathbf{k}}{\sqrt{2}}
\mid
\varepsilon_2,\varepsilon_3\in\{\pm1\}
\right\}.
\]
We denote this eight-element antipodal set by
\begin{equation}
\DD_0=\tau\QQ.
\end{equation}
Equivalently,
\[
\DD_0
=
\left\{
\pm\frac{1+\mathbf{i}}{\sqrt{2}},
\pm\frac{1-\mathbf{i}}{\sqrt{2}},
\pm\frac{\mathbf{j}+\mathbf{k}}{\sqrt{2}},
\pm\frac{\mathbf{j}-\mathbf{k}}{\sqrt{2}}
\right\}.
\]
Thus $\DD_0$ is precisely the subset of $\DD$ whose elements are supported
either on the coordinate pair $\{1,\mathbf{i}\}$ or on the coordinate
pair $\{\mathbf{j},\mathbf{k}\}$.

Now let
\[
c
=
\frac{
\varepsilon_0
+\varepsilon_1\mathbf{i}
+\varepsilon_2\mathbf{j}
+\varepsilon_3\mathbf{k}
}{2}
\in \CC,
\]
where each $\varepsilon_r$ belongs to $\{\pm1\}$. Direct multiplication
gives
\[
\begin{aligned}
\tau c
&=
\frac{1+\mathbf{i}}{\sqrt{2}}
\cdot
\frac{
\varepsilon_0
+\varepsilon_1\mathbf{i}
+\varepsilon_2\mathbf{j}
+\varepsilon_3\mathbf{k}
}{2}
=
\frac{1}{2\sqrt{2}}
\Bigl(
(\varepsilon_0-\varepsilon_1)
+
(\varepsilon_0+\varepsilon_1)\mathbf{i}
+
(\varepsilon_2-\varepsilon_3)\mathbf{j}
+
(\varepsilon_2+\varepsilon_3)\mathbf{k}
\Bigr).
\end{aligned}
\]
Exactly one of
\[
\varepsilon_0-\varepsilon_1,
\qquad
\varepsilon_0+\varepsilon_1
\]
is zero, while the other is equal to $2$ or $-2$. Similarly, exactly
one of
\[
\varepsilon_2-\varepsilon_3,
\qquad
\varepsilon_2+\varepsilon_3
\]
is zero, while the other is equal to $2$ or $-2$.

It follows that $\tau c$ has exactly two nonzero coordinates, each
equal to $1/\sqrt{2}$ or $-1/\sqrt{2}$. 
One nonzero coordinate belongs to the pair
$\{1,\mathbf{i}\}$, while the other belongs to the pair
$\{\mathbf{j},\mathbf{k}\}$. Hence the possible supports are
\[
\{1,\mathbf{j}\},
\qquad
\{1,\mathbf{k}\},
\qquad
\{\mathbf{i},\mathbf{j}\},
\qquad
\{\mathbf{i},\mathbf{k}\}.
\]

We divide these elements into two antipodal subsets. Define
\begin{equation}
\DD_1
=
\left\{
\frac{\varepsilon_0+\varepsilon_2\mathbf{j}}{\sqrt{2}}
\mid
\varepsilon_0,\varepsilon_2\in\{\pm1\}
\right\}
\cup
\left\{
\frac{\varepsilon_1\mathbf{i}+\varepsilon_3\mathbf{k}}{\sqrt{2}}
\mid
\varepsilon_1,\varepsilon_3\in\{\pm1\}
\right\}.
\end{equation}
Equivalently,
\[
\DD_1
=
\left\{
\pm\frac{1+\mathbf{j}}{\sqrt{2}},
\pm\frac{1-\mathbf{j}}{\sqrt{2}},
\pm\frac{\mathbf{i}+\mathbf{k}}{\sqrt{2}},
\pm\frac{\mathbf{i}-\mathbf{k}}{\sqrt{2}}
\right\}.
\]
Thus $\DD_1$ consists of the elements supported on the coordinate pairs
$\{1,\mathbf{j}\}$ and $\{\mathbf{i},\mathbf{k}\}$.

Similarly, define
\begin{equation}
\DD_2
=
\left\{
\frac{\varepsilon_0+\varepsilon_3\mathbf{k}}{\sqrt{2}}
\mid
\varepsilon_0,\varepsilon_3\in\{\pm1\}
\right\}
\cup
\left\{
\frac{\varepsilon_1\mathbf{i}+\varepsilon_2\mathbf{j}}{\sqrt{2}}
\mid
\varepsilon_1,\varepsilon_2\in\{\pm1\}
\right\}.
\end{equation}
Equivalently,
\[
\DD_2
=
\left\{
\pm\frac{1+\mathbf{k}}{\sqrt{2}},
\pm\frac{1-\mathbf{k}}{\sqrt{2}},
\pm\frac{\mathbf{i}+\mathbf{j}}{\sqrt{2}},
\pm\frac{\mathbf{i}-\mathbf{j}}{\sqrt{2}}
\right\}.
\]
Thus $\DD_2$ consists of the elements supported on the coordinate pairs
$\{1,\mathbf{k}\}$ and $\{\mathbf{i},\mathbf{j}\}$.

More explicitly, the location of $\tau c$ is determined by the signs as follows:
\[
\begin{array}{c|c|c}
\text{Condition on }\varepsilon_0,\varepsilon_1
&
\text{Condition on }\varepsilon_2,\varepsilon_3
&
\text{Support of }\tau c
\\
\hline
\varepsilon_0=-\varepsilon_1
&
\varepsilon_2=-\varepsilon_3
&
\{1,\mathbf{j}\}
\\
\varepsilon_0=-\varepsilon_1
&
\varepsilon_2=\varepsilon_3
&
\{1,\mathbf{k}\}
\\
\varepsilon_0=\varepsilon_1
&
\varepsilon_2=-\varepsilon_3
&
\{\mathbf{i},\mathbf{j}\}
\\
\varepsilon_0=\varepsilon_1
&
\varepsilon_2=\varepsilon_3
&
\{\mathbf{i},\mathbf{k}\}.
\end{array}
\]
Consequently,
\[
\tau \CC_1
=
\DD_1, \tau \CC_2
=
\DD_2.
\]
Together with $\tau\QQ=\DD_0$, this gives
\[
\tau \CC
=
\DD_0\sqcup \DD_1\sqcup \DD_2
=
\DD.
\]
Thus, we have proved that
$\DD=\tau \CC$.
Together with $\tau \CC=\CC\tau$, this gives
$\DD=\tau \CC=\CC\tau$.

Finally, using $\CC\CC=\CC$, $\CC\tau=\tau \CC$, and
$\tau^2=\mathbf{i}\in \CC$, we obtain
$\CC\DD=\DD$,
and
$\DD\CC
=\DD$.
Moreover,
$\DD\DD
=\CC$.
Thus $2O=\CC\cup \DD$ is closed under multiplication.

Every element of $2O$ is a unit quaternion, so its inverse is its
quaternionic conjugate. The sets $\CC$ and $\DD$ are both invariant under
quaternionic conjugation, and hence $2O$ is also closed under inversion.
Therefore $2O$ is a group. Since $\CC$ and $\DD$ are disjoint and each has
cardinality $24$, we have
\[
|2O|=48.
\]
\end{proof}

\begin{remark}
Since $\tau$ is a unit quaternion, the map
\[
L_\tau\colon\HH\longrightarrow\HH,
\qquad
q\longmapsto\tau q,
\]
is an orthogonal transformation of $\RR^4$. Therefore, the identity
$\DD=\tau \CC$
 shows directly that $\CC$ and $\DD$ are congruent realizations
of the regular $24$-cell.

The multiplication rules
\[
\CC\CC=\CC,
\qquad
\CC\DD=\DD\CC=\DD,
\qquad
\DD\DD=\CC
\]
have the following geometric meaning: multiplication by an element of
$\CC$ preserves each of the two $24$-cells, whereas multiplication by an
element of $\DD$ interchanges them.

Thus, $\CC$ is a normal subgroup of index $2$ in $2O$, and $\DD$ is the
nontrivial coset of $\CC$. The group $2O$ is the binary octahedral group.
\end{remark}

\section{The ternary relations}
\label{sec:canonical-relations}

\subsection{Sign lifts}
To summarize the previous section, for $r\in\ZZ_3\eqdef\ZZ/3\ZZ$, we have
\[
    \CC_r=\omega^r\QQ,
    \qquad
    \DD_r=\tau\omega^r\QQ.
\]
and
\[
    \CC=\CC_0\sqcup\CC_1\sqcup\CC_2,
    \qquad
    \DD=\DD_0\sqcup\DD_1\sqcup\DD_2.
\]

Recall that $\QQ$ is an abelian group w.r.t. multiplication and $\{1,-1\}$ is its normal subgroup. Multiplication on $\QQ$ induces an operation on the factor-group $\QQ/\{\pm1\}$, which we denote by $+$ for convenience. In fact, we have
\[
    K=\QQ/\{+1,-1\}\cong\FF_2^2.
\]
An element of the factor-group that contains an element $a$ is denoted by $\ol{a}$. Let us also introduce a notation: $q: K\longrightarrow \{1,{\mathbf i},{\mathbf j},{\mathbf k}\}$
where
\[
q_{\ol 1}=1,\qquad
q_{\ol{\mathbf i}}={\mathbf i},\qquad
q_{\ol{\mathbf j}}={\mathbf j},\qquad
q_{\ol{\mathbf k}}={\mathbf k}.
\]
Note that $q$ is simply a mapping that chooses a representative in every factor-class.

\begin{definition}
Define the many-valued map
\[
    \Lift:K\longrightarrow\Pow{\QQ},
    \qquad
    \Lift(u)=u, \text{ or }\Lift(u)=\{q_{u},-q_{u}\}.
\]
For $S\subseteq K^3$, define its coordinate-wise sign lift by
\[
    \widehat S
    \eqdef
    \bigcup_{(u,v,w)\in S}
    \Lift(u)\times\Lift(v)\times\Lift(w)
    \subseteq \QQ^3.
\]
\end{definition}
A key element of the quaternionic construction is a sign lift of the following ternary relation
\[
    \BB_0=\{(u,v,w)\in K^3\mid u+v+w=0\},
\]
which can be seen as a linear subspace of ${\rm GF}(2^2)^3$, and corresponding affine subspaces of the form
\[
    \BB_\eta
    =
    \{(u,v,w)\in K^3\mid u+v+w=\eta\} = (0,0,\eta)+\BB_0,
\]
for every value $\eta\in K$. Thus the four ternary relations are
$\BB_0$, $\BB_{\ol{\mathbf i}}$, $\BB_{\ol{\mathbf j}}$, $\BB_{\ol{\mathbf k}}$.

Every $\BB_\eta$ is a translate of $\BB_0$ and has $|K|^2=16$
elements. Its sign lift is
\[
    \widehat{\BB}_\eta
    =
    \bigcup_{u,v\in K}
    \Lift(u)\times\Lift(v)\times\Lift(\eta+u+v),
\]
and therefore, 
$|\widehat{\BB}_\eta|=4^2\cdot 2^3=128$.

\subsection{Quaternionic rotations of the lifted ternary relations}

For unit quaternions $r_1,r_2,r_3$ and $S\subseteq \QQ^3$, let us define
\[
    (r_1,r_2,r_3)[S]
    =
    \{(r_1a,r_2b,r_3c)\mid (a,b,c)\in S\}.
\]
All multiplications in this notation are left quaternionic multiplications.
They are orthogonal transformations of $\HH\cong\mathbb R^4$, which makes $(r_1,r_2,r_3)$ act unitarily on ${\mathbb R}^{12}$.

\begin{definition}[First ternary relation]
For $r\in\ZZ_3$, define
\[
    \Tau_{1,r}
    =
    (\tau\omega^r,\omega^r,\omega^r)
    [\widehat{\BB}_0] 
    \subseteq\DD_r\times\CC_r\times\CC_r.
\]
Their disjoint union (due to disjointness of $\DD_r\times\CC_r\times\CC_r, r\in\ZZ_3 $) is
\[
\begin{aligned}
    \Tau_1
    ={}&
    (\tau,1,1)[\widehat{\BB}_0]
    \sqcup
    (\tau\omega,\omega,\omega)[\widehat{\BB}_0]
    \sqcup
    (\tau\omega^2,\omega^2,\omega^2)[\widehat{\BB}_0].
\end{aligned}
\]
\end{definition}
Let 
\begin{equation}
\delta_0=\ol{\mathbf j}, \delta_1=\ol{\mathbf i}, \delta_2=\ol{\mathbf k}.
\end{equation}
\begin{definition}[Second ternary relation] 
For $r\in\ZZ_3$, define
\[
    \Tau_{2,r}
    =
    (\omega^r,\tau\omega^r,\tau\omega^r)
    [\widehat{\BB}_{\delta_r}]
    \subseteq\CC_r\times\DD_r\times\DD_r.
\]
Their disjoint union is
\[
\begin{aligned}
    \Tau_2
    ={}&
    (1,\tau,\tau)[\widehat{\BB}_{\ol{\mathbf j}}]
    \sqcup
    (\omega,\tau\omega,\tau\omega)
       [\widehat{\BB}_{\ol{\mathbf i}}]
    \sqcup
    (\omega^2,\tau\omega^2,\tau\omega^2)
       [\widehat{\BB}_{\ol{\mathbf k}}].
\end{aligned}
\]
\end{definition}

Thus the quaternionic relations are obtained by the two-step procedure
\[
    \BB_\eta
    \longmapsto
    \widehat{\BB}_\eta\subseteq \QQ^3
    \longmapsto
    (r_1,r_2,r_3)[\widehat{\BB}_\eta]\subseteq\HH^3.
\]
The first family uses three copies of the homogeneous parity $\BB_0$.
The second family uses the three shifted
$\BB_{\delta_0},\BB_{\delta_1},\BB_{\delta_2}$.

\begin{lemma}
\label{prop:tau-cardinality}
We have
$|\Tau_1|=|\Tau_2|=384$.
\end{lemma}

\begin{proof}
For every $r\in\ZZ_3$, we have
\[
    |\Tau_{1,r}|=|\Tau_{2,r}|
    =|\widehat{\BB}_\eta|=128.
\]
Therefore, a union over $r\in\ZZ_3$ gives $3\cdot 128=384$ elements.
\end{proof}

\section{The quaternionic construction of $840$-point kissing arrangement}
\label{sec:canonical-arrangement}
Recall that we
identify $\mathbb R^{12}$ with
$\HH^3$. Define the axial families
\[
    \EE_1=\{(c,0,0)\mid c\in\CC\},
    \qquad
    \EE_2=\{(0,c,0)\mid  c\in\CC\},
    \qquad
    \EE_3=\{(0,0,c)\mid  c\in\CC\}.
\]
These are three $24$-cells supported on the first, second and third $4$-dimensional factors of $\mathbb R^{12}$.
Define the mixed families
\[
    \MM_1
    =
    \left\{
       \left(\frac a{\sqrt2},\frac b2,\frac c2\right)\mid
       (a,b,c)\in\Tau_1
    \right\},
\]
and
\[
    \MM_2
    =
    \left\{
       \left(\frac a2,\frac b{\sqrt2},\frac c2\right)\mid
       (a,b,c)\in\Tau_2
    \right\}.
\]
Finally, put
\[
    \XX_0
    =
    \EE_1\sqcup\EE_2\sqcup\EE_3
    \sqcup\MM_1\sqcup\MM_2.
\]
In the lifting notation, the mixed pieces are
\[
\begin{aligned}
    \MM_{1,r}
    &=
    \left(
       \frac{\tau\omega^r}{\sqrt2},
       \frac{\omega^r}{2},
       \frac{\omega^r}{2}
    \right)[\widehat{\BB}_0],
    \\
    \MM_{2,r}
    &=
    \left(
       \frac{\omega^r}{2},
       \frac{\tau\omega^r}{\sqrt2},
       \frac{\tau\omega^r}{2}
    \right)[\widehat{\BB}_{\delta_r}],
\end{aligned}
\]
with
\[
    \MM_1=\bigsqcup_r\MM_{1,r},
    \qquad
    \MM_2=\bigsqcup_r\MM_{2,r}.
\]

\begin{theorem}
\label{thm:canonical-kissing}
The set $\XX_0$ is an antipodal $840$-point kissing arrangement in
$\mathbb R^{12}$. In other words, all its points are unit vectors and
\[
    \ip{x}{y}\le\frac12
\]
for all distinct $x,y\in\XX_0$.
\end{theorem}

\begin{remark} Note that the squared norm of any vector in $\widehat{\BB}_{\eta}\subseteq \HH^3$ receives an equal contribution from each quaternionic factor $\HH$, namely $1$. Multiplication by a unit quaternion does not change the norm. After the transformations defining $\MM_1$ and $\MM_2$, every vector has the squared norm $1$, but the contributions of the three quaternionic factors $\HH$ are no longer equal. They are
$\left(\frac12,\frac14,\frac14\right)$
for vectors in $\MM_1$, and
$\left(\frac14,\frac12,\frac14\right)$
for vectors in $\MM_2$. In particular, in both cases the contribution from the third quaternionic factor is $\frac14$, which is smaller than the average contribution $\frac38$ from the first two factors. In this sense, the construction ``frees'' some space in the third $\HH$-factor. It is precisely this freed space that makes it possible to squeeze in an additional $841$-st vector, as discussed in Section~\ref{sec:comparison-canonical-840}. 

Unfortunately, freeing some space in the third factor is not by itself sufficient to accommodate an additional vector. The $841$-point kissing arrangement also slightly deforms the vectors in $\EE_1$, $\EE_2$, $\MM_1$, and $\MM_2$, while deforming the vectors in $\EE_3$ much more substantially.
\end{remark}

The proof of Theorem~\ref{thm:canonical-kissing}, together with the verification that this arrangement is distinct from all previously known $840$-point configurations, can be found in the Appendix.

\section{Comparison with the numerical $841$-point kissing arrangement}
\label{sec:comparison-canonical-840}

We next compare the numerical $841$-point kissing arrangement from~\cite{takhanov2026structurekissingarrangementsmathbb} (which can be found in the corresponding Github repository) with the quaternionic $840$-point arrangement that we describe in the previous section.  Let
\[
\mathcal{X}_{841}=\{\widetilde{x}_1,\ldots,\widetilde{x}_{841}\}
\subset \mathbb{S}^{11}
\]
denote the numerical configuration.  As described above, one point is
intrinsically distinguished from the remaining configuration; in the supplied
coordinate file it is the vector with index $751$.  We remove this point and
write
\[
\mathcal{X}_{840}=\mathcal{X}_{841}\setminus\{\widetilde{x}_{751}\}
   =\{x_1,\ldots,x_{840}\}.
\]
Let
\[
\mathcal{Y}_{840}=\{y_1,\ldots,y_{840}\}
\]
be the quaternionic arrangement, ordered as
\[
\mathcal{Y}_{840}
 =\underbrace{\mathcal{E}_1\sqcup\mathcal{E}_2
 \sqcup\mathcal{M}_1\sqcup\mathcal{M}_2}_{816\text{ vectors}}
 \sqcup
 \underbrace{\mathcal{E}_3}_{24\text{ vectors}}.
\]
Here $\mathcal{E}_1,\mathcal{E}_2,\mathcal{E}_3$ are the three axial regular
$24$-cells, and $\mathcal{M}_1,\mathcal{M}_2$ are the two mixed
$384$-point families.  In particular, $y_1,\ldots,y_{816}$ form the
quaternionic antipodal backbone together with the first two axial $24$-cells,
whereas $y_{817},\ldots,y_{840}$ are precisely the third axial $24$-cell
$\mathcal{E}_3=\{(0,0,c)\mid c\in C\}$,
which lies in the last quaternionic coordinate space, i.e. in
$\{0\}\times \{0\}\times\mathbb{H}$.

To compare the two configurations independently of their coordinate frames and
labelings, we solved the problem
\begin{equation}
\label{eq:assignment-procrustes-840}
\min_{\pi\in S_{840},\ U\in O(12)}
\sum_{i=1}^{840}
\left\|x_i-Uy_{\pi(i)}\right\|_2^2.
\end{equation}
For fixed $U$, the optimal permutation is obtained by solving the linear
assignment problem with costs
$\|x_i-Uy_j\|_2^2$.  For fixed $\pi$, the optimal orthogonal matrix is obtained
from the orthogonal Procrustes problem.  We alternated these two exact updates
from several structure-based initializations.  Since
\eqref{eq:assignment-procrustes-840} is nonconvex, the value reported below is
the best value found numerically rather than a proof of the global minimum.

The best alignment gave
\[
\sum_{i=1}^{840}\left\|x_i-Uy_{\pi(i)}\right\|_2^2
 =11.96241654955465,
\]
corresponding to a root-mean-square Euclidean discrepancy of
\[
\left(
\frac1{840}\sum_{i=1}^{840}
\left\|x_i-Uy_{\pi(i)}\right\|_2^2
\right)^{1/2}
 =0.119335544088101.
\]

For each numerical point define the squared matching residual
\[
\rho_i=\left\|x_i-Uy_{\pi(i)}\right\|_2^2,
\]
and let
\[
\rho_{(1)}\leq\rho_{(2)}\leq\cdots\leq\rho_{(840)}
\]
be their order statistics.  Figure~\ref{fig:canonical840-order-statistics}
plots $i$ against $\rho_{(i)}$.

\begin{figure}[htbp]
    \centering
    \includegraphics[width=0.88\textwidth]
    {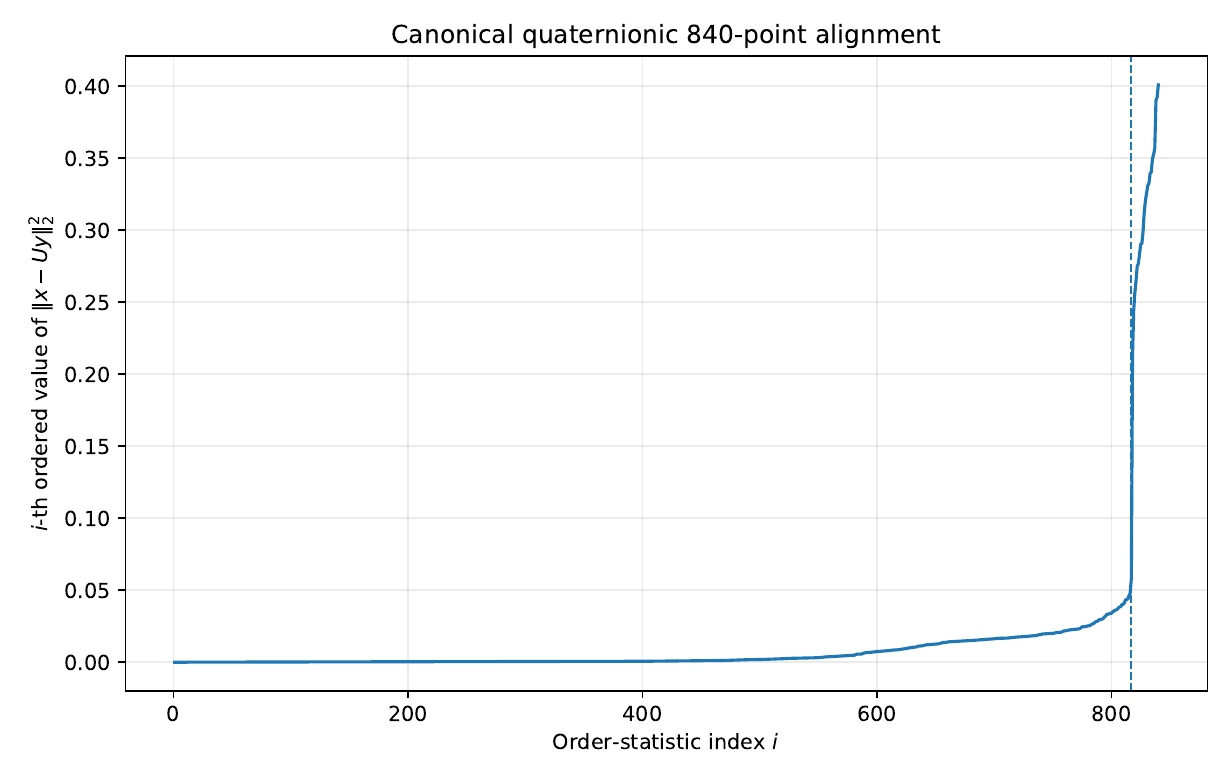}
    \caption{Order statistics of the squared matching residuals
    $\|x_i-Uy_{\pi(i)}\|_2^2$ for the best permutation--orthogonal transformation
    between the $840$-point subsystem of the numerical $841$-point arrangement
    and the quaternionic $840$-point arrangement.  The vertical
    dashed line separates the first $816$ residuals from the largest $24$.}
    \label{fig:canonical840-order-statistics}
\end{figure}

The residuals exhibit a sharp structural separation.  All canonical vectors
$y_1,\ldots,y_{816}$ are matched among the first $816$ order statistics, while
the $24$ largest residuals correspond exactly to
$y_{817},\ldots,y_{840}\in\mathcal{E}_3$.  Quantitatively, the contribution of
the first $816$ canonical vectors is
\[
\sum_{\pi(i)\leq816}\rho_i
 =4.66601501226742,
\qquad
\left(
\frac1{816}\sum_{\pi(i)\leq816}\rho_i
\right)^{1/2}
 =0.075618487503312,
\]
and their largest squared residual is
\[
\max_{\pi(i)\leq816}\rho_i
 =\rho_{(816)}
 =0.0479262610331176.
\]
By contrast, the third axial $24$-cell contributes
\[
\sum_{\pi(i)\geq817}\rho_i
 =7.29640153728723,
\qquad
\left(
\frac1{24}\sum_{\pi(i)\geq817}\rho_i
\right)^{1/2}
 =0.551377122050146,
\]
with
\[
\rho_{(817)}=0.0580657132382830,
\qquad
\rho_{(840)}=0.400924978439449.
\]
Thus, after a single orthogonal change of coordinates and a relabeling, the $816$-point cores of the two arrangements are very close, while nearly all of the visible deformation is concentrated in the remaining $24$ points.

Restoring the deleted special point
$\widetilde{x}_{751}$ therefore gives the structural decomposition
\[
\mathcal{X}_{841}
 =\underbrace{\mathcal{B}_{816}}_{\text{small deformation of the canonical core}}
 \sqcup
 \underbrace{\mathcal{D}_{24}}_{\text{deformation of }\mathcal{E}_3}
 \sqcup
 \underbrace{\{\widetilde{x}_{751}\}}_{\text{special additional point}},
\]
where $\mathcal{D}_{24}$ denotes the $24$ numerical vectors matched to
$\mathcal{E}_3$.  In this sense, the exceptional $25$-point part of the
$841$-point arrangement is naturally interpreted as a deformed regular
$24$-cell together with one additional distinguished vector, while the other
$816$ vectors remain close to the quaternionic arrangement.

All scripts generating plots and data can be found in the \href{https://github.com/k-nic/quaternionic}{Github} repository accompanying the paper.

\section{A general notion of quaternionic construction}\label{quaternionic-general}
We now introduce a general class of quaternionic configurations in dimensions
divisible by $4$. Let
$\RR^{4k}\cong \HH^k$.
For $j=1,\ldots,k$, define the $j$-th axial $24$-cell by
\[
    \Ecal_j
    =
    \{0\}^{j-1}\times \CC\times \{0\}^{k-j}
    \subseteq \HH^k.
\]
Thus, the construction contains $k$ mutually orthogonal $24$-cells,
$\Ecal_1,\ldots,\Ecal_k$,
one in each quaternionic coordinate factor.

Let
\[
    w=(w_1,\ldots,w_k),
    \qquad
    w_j\ge 0,
    \qquad
    \sum_{j=1}^k w_j^2=1,
\]
Given a
relation
$\Rcal\subseteq (2O)^k$,
define its weighted quaternionic realization by
\[
    w\Rcal
    \eqdef
    \bigl\{
        (w_1q_1,\ldots,w_kq_k)\mid
       (q_1,\ldots,q_k)\in\Rcal
    \bigr\}.
\]

\begin{definition}\label{def:quaternionic-construction}
A \emph{quaternionic construction} in
$\RR^{4k}\cong\HH^k$ is a configuration of the form
\[
    \Xcal
    =
    \bigcup_{j=1}^k \Ecal_j
    \;\cup\;
    \bigcup_{\alpha} w^{(\alpha)}\Rcal_{\alpha},
\]
where
\[
    \Rcal_{\alpha}\subseteq (2O)^k
\]
is a relation specifying the admissible quaternionic labels, and
\[
    w^{(\alpha)}
    =
    \bigl(
        w_1^{(\alpha)},\ldots,w_k^{(\alpha)}
    \bigr),
    \qquad
    \sum_{j=1}^k
    \bigl(w_j^{(\alpha)}\bigr)^2=1,
\]
is a corresponding weight vector.
\end{definition}

Thus, a quaternionic construction consists of 
$k$ mutually orthogonal $24$-cells, together with one or more weighted
families of mixed vectors whose quaternionic coordinates belong to $2O$. The
weight vectors describe how the squared norm of each mixed family is
distributed among the $k$ quaternionic factors.

The notion of a quaternionic construction encompasses several of known kissing arrangements in dimensions divisible by $4$. In
particular, it includes
\begin{enumerate}
\item the regular $24$-cell in $\mathbb R^4$, consisting of $24$ points~\cite{delaat2024optimalityuniquenessd4root};
\item the $E_8$ kissing arrangement in $\mathbb R^8$, consisting of $240$ points;
\item the $840$-point kissing arrangement in $\mathbb R^{12}$ constructed above;
\item of $1579$ kissing arrangements of size $840$ in $\mathbb R^{12}$ corresponding to maximum independent sets in signed Johnson graphs, $68$ are antipodal. Of those $68$ antipodal configurations, $64$ have quaternionic constructions. This includes $17$ configurations, which can be obtained from Leech-Slone lifting~\cite{Leech_Sloane_1971} of $17$ non-isomorphic binary $(12,4,4)$-codes~\cite{Best1978};
\item the $4320$-point kissing arrangement of the Barnes--Wall lattice in $\mathbb R^{16}$~\cite{Barnes_Wall_1959};
\item the $19448$-point kissing arrangement in $\mathbb R^{20}$ constructed by Cohn and Li~\cite{cohn2026improvedkissingnumbersseventeen};
\end{enumerate}
Thus, the quaternionic construction captures a sequence of remarkable
configurations in
\[
    \mathbb R^4,\ \mathbb R^8,\ \mathbb R^{12},\
    \mathbb R^{16},\ \mathbb R^{20}.
\]
We give an explicit verification for the $E_8$ configuration below, since
in this case the argument is particularly short. We discuss
other configurations in the following subsection.
\begin{theorem}
\label{prop:E8-quaternionic}
Identify
$\RR^8\cong\HH^2$. 
Let
\[
    \Rcal_{E_8}
    \eqdef
    (\mathcal{D}_0\times \mathcal{D}_0)
    \cup(\mathcal{D}_1\times \mathcal{D}_1)
    \cup(\mathcal{D}_2\times \mathcal{D}_2)
    \subseteq(2O)^2.
\]
Then, the configuration
\[
    \Xcal_{E_8}
    =
    (\CC\times\{0\})
    \cup
    (\{0\}\times\CC)
    \cup
    \frac1{\sqrt2}\Rcal_{E_8}
\]
is the normalized $E_8$ root system. In particular,
$E_8$ is a quaternionic construction with $k=2$, two axial
$24$-cells, and mixed weight vector
\[
    w=\left(\frac1{\sqrt2},\frac1{\sqrt2}\right).
\]
\end{theorem}

\begin{proof}
It is straightforward to check, using Lemma~\ref{lem:coset-scalar-products}, that $\Xcal_{E_8}$ is a kissing arrangement of size $240$. Since $E_8$ root system is uniquely defined by this property, then it is isometric to $\Xcal_{E_8}$.
\end{proof}

\subsection{An algorithm for checking quaternionic structure: the Barnes--Wall and Cohn--Li configurations}\label{barnes}

In the preceding section we showed that the $E_8$ kissing arrangement admits natural quaternionic description.
The same phenomenon persists in dimensions $16$ and $20$: the Barnes--Wall
kissing arrangement in $\mathbb{R}^{16}$ and the Cohn--Li $19448$-point
arrangement in $\mathbb{R}^{20}$ also admit realizations built from copies of
the regular $24$-cell and the binary octahedral group $2O$.
For these two configurations, however, a direct structural proof is longer and less transparent.  We therefore
use an algorithmic approach: we search through orthogonal collections of
$24$-cells, convert each such collection into quaternionic coordinates, and
retain only those decompositions for which the remaining vectors have the
expected $2O$ structure.

Let $I\subset S^{d-1}$ be any kissing arrangement, where $d=4k$.  We search for $k$ mutually orthogonal
$4$-dimensional subspaces spanned by regular $24$-cells contained in $I$.

For each detected $24$-cell $u$ in $I$, let $H_u$ denote its $4$-dimensional
linear span.  Choose an orthogonal identification of $H_u$ with
$\mathbb H$ that maps $u$ to the standard $24$-cell $\mathcal C$.
We then test the projection of the entire configuration onto $H_u$.
More precisely, for every nonzero vector in
$\operatorname{Proj}_{H_u} I$, 
its normalized direction is required to belong to the binary octahedral
group $2O$.  

Only a $24$-cell that passes this projection test is retained as a vertex
of the graph $G$.  Two retained vertices $u$ and $v$ are adjacent when
\[
        H_u\perp H_v.
\]
Thus, a $k$-clique
\[
        \{u_1,\ldots,u_k\}
\]
determines an orthogonal decomposition
\[
        \mathbb R^{4k}
        =
        H_{u_1}\oplus\cdots\oplus H_{u_k}.
\]
The graph is constructed incrementally: whenever a new admissible
$24$-cell is found, it is joined to the previously found admissible
$24$-cells having orthogonal spans, and we immediately test whether it
completes a $k$-clique.  Hence there is no need to construct the graph of
all $24$-cells in advance or to enumerate all its $k$-cliques.

Given such a $k$-clique, we choose an orthogonal transformation
$T\in O(4k)$ mapping its $24$-cells to the standard coordinate copies
\[
\begin{aligned}
&\mathcal C\times 0\times\cdots\times 0,\quad
0\times\mathcal C\times 0\times\cdots\times 0,\quad\ldots,\quad
0\times\cdots\times 0\times\mathcal C .
\end{aligned}
\]
Then, the required binary-octahedral condition for a
quaternionic construction is satisfied by construction. 


The search used in our computations is summarized in
Algorithm~\ref{alg:quaternionic-search}.  Using this algorithm we verified
that the Barnes--Wall and the Cohn--Li arrangements admit quaternionic
constructions.

\begin{algorithm}[ht]
\caption{Search for a quaternionic structure}
\label{alg:quaternionic-search}
\begin{algorithmic}[1]
\Require A kissing arrangement $I\subset S^{4k-1}$, the standard $24$-cell
$\mathcal C\subset\mathbb H$, and the binary octahedral group $2O$
\Ensure A $k$-clique defining a quaternionic decomposition, if one is found

\State Extract the antipodal lines contained in $I$
\State Select the antipodal lines to be used as possible frame lines
\State Initialize the graph $G$ with no vertices

\For{each antipodal line of $I$}
    \State Find possible frame lines making inner product of absolute value
    $1/2$ with the current line
    \State Enumerate mutually orthogonal $4$-tuples among these frame lines

    \For{each such orthogonal $4$-tuple}
        \State Check whether the corresponding $12$ antipodal lines of a regular $24$-cell occur in $I$

        \If{the $24$-cell is present}
            \State Identify its span with $\mathbb H$ so that the
            $24$-cell is mapped to $\mathcal C$

            \State Test whether every nonzero projection of $I$ onto this
            span has normalized direction in $2O$

            \If{the projection test succeeds}
                \State Add the corresponding $4$-dimensional span to $G$
                if it has not already been retained

                \State Join the new vertex to every previously retained
                vertex having an orthogonal span

                \State Search for a $k$-clique containing the new vertex

                \If{such a $k$-clique $\{u_1,\ldots,u_k\}$ is found}
                    \State Construct $T\in O(4k)$ mapping
                    $u_1,\ldots,u_k$ to the $k$ coordinate copies of
                    $\mathcal C$

                   \State \Return $\{u_1,\ldots,u_k\}$ and $T$
                   
                \EndIf
            \EndIf
        \EndIf
    \EndFor
\EndFor

\State \Return no decomposition found
\end{algorithmic}
\end{algorithm}

An optimal kissing arrangement cannot be given by a quaternionic construction already for $k=6$, i.e. $d=24$. We applied Algorithm~\ref{alg:quaternionic-search} to the kissing arrangement in ${\mathbb R}^{24}$, defined by the Leech lattice, and it gave a negative result. 
The Algorithm also showed nonexistence of a quaternionic construction for $204520$-point record kissing arrangement in ${\mathbb R}^{28}$ found by the PackingStar system~\cite{ma2026findingkissingnumbersgametheoretic}. The latter is not surprising as the construction involves a certain lifting of Leech lattice vectors to ${\mathbb R}^{28}$~\cite{cohn2011rigidity,doi:10.1137/16M1095810}).

\bibliographystyle{IEEEtran}
\bibliography{lit}

\appendix

\section*{Proof of Theorem~\ref{thm:canonical-kissing}}\label{main-proof}
\subsection{Scalar products of elements of the two $24$-cells}

\begin{lemma}
\label{lem:coset-scalar-products}
The following statements hold.
\begin{enumerate}
\item If $x,y$ belong to the same coset $\CC_r$ or the same coset $\DD_r$,
then
\[
    \ip{x}{y}\in\{-1,0,1\},
\]
and the scalar product is nonzero exactly when $\{x, -x\}=\{y, -y\}$.
\item If $r\ne s$, then
\[
    \ip{x}{y}\in\left\{-\frac12,\frac12\right\}
\]
for $x\in\CC_r,y\in\CC_s$, and likewise for
$x\in\DD_r,y\in\DD_s$.
\item If $x\in\CC$ and $y\in\DD$, then
\[
    \ip{x}{y}\in
    \left\{-\frac1{\sqrt2},0,\frac1{\sqrt2}\right\}.
\]
\end{enumerate}
\end{lemma}

\begin{proof}
Every coset has the form $q\QQ$ for a unit quaternion $q$. Multiplication by $q$ is an orthogonal transformation on $\HH$, while the four elements
$1,\mathbf i,\mathbf j,\mathbf k$ of $\QQ$ are mutually orthogonal. This proves (i).

For any two  $x\in\CC_r,y\in\CC_s$,  $r\ne s$,
either both $x$ and $y$ have the form $\frac{1}{2}(\pm 1\pm {\mathbf i}\pm {\mathbf j}\pm {\mathbf k})$ with sign patterns of different parity, or one of them is of the form $\frac{1}{2}(\pm 1\pm {\mathbf i}\pm {\mathbf j}\pm {\mathbf k})$ and another is from $\QQ$.
This proves the first assertion in (ii). The assertion for $\DD$ follows
because multiplication by $\tau$ is an orthogonal transformation. 

Finally, every element
of $\tau\CC$ has the form $\frac{1}{\sqrt2}(\pm a \pm b)$ where $a,b\in \{1,\mathbf i,\mathbf j,\mathbf k\}$ are distinct, which gives (iii).
\end{proof}

For $m\in\ZZ_3$, define the nonzero linear functionals
$\ell_m:K\to\FF_2$ by
\[
\begin{aligned}
    \ell_0(x_1,x_2)&=x_2,
    \\
    \ell_1(x_1,x_2)&=x_1,
    \\
    \ell_2(x_1,x_2)&=x_1+x_2.,
\end{aligned}
\]
where elements of $K$ are identified with pairs of $\FF$-elements.
\begin{lemma}
\label{lem:cross-incidence}
Let $u,v\in K$, and let
\[
c\in\omega^r\Lift(u)\subseteq\CC_r,
\qquad
d\in\tau\omega^s\Lift(v)\subseteq\DD_s.
\]
Then
\[
\ip{c}{d}\ne0
\quad\Longleftrightarrow\quad
\ell_{r+s}(u+v+\delta_r)=1,
\]
where the index $r+s$ is taken modulo $3$.
\end{lemma}

\begin{proof}
Changing either sign changes only the sign of the scalar product, not whether
it is zero. Hence
\[
    \ip{c}{d}\ne0
    \quad\Longleftrightarrow\quad
    \ip{\omega^r q_u}{\tau\omega^s q_v}\ne0.
\]

We have
\[
\begin{aligned}
    \ip{\omega^r q_u}{\tau\omega^s q_v}
    &=
    \RePart\!\left(
        \omega^r q_u
        \ol{q_v}\,
        \ol{\omega}^{\,s}
        \ol{\tau}
    \right).
\end{aligned}
\]
The addition operation in $K=\QQ/\{\pm1\}$ is induced by the multiplication in $\QQ$.  Therefore, the factor-class of
$q_u\ol{q_v}$ is $u+v$.
Consequently there is a sign
$\sigma(u,v)\in\{+1,-1\}$ such that
$q_u\ol{q_v}
    =
    \sigma(u,v)\,q_{u+v}$.
Again this sign does not affect vanishing of the scalar product.  Set
$z=u+v$.
Then
\[
    \ip{\omega^r q_u}{\tau\omega^s q_v}\ne0
    \quad\Longleftrightarrow\quad
    \RePart\!\left(
        \omega^r q_z
        \ol{\omega}^{\,s}
        \ol{\tau}
    \right)\ne0.
\]
Using $\RePart(ab)=\RePart(ba)$, we obtain
\[
    \RePart\!\left(
        \omega^r q_z
        \ol{\omega}^{\,s}
        \ol{\tau}
    \right)
    =
    \RePart\!\left(q_z A_{r,s}\right),
    \qquad
    A_{r,s}\eqdef
    \ol{\omega}^{\,s}\ol{\tau}\omega^r.
\]
We now compute the nine quaternions $A_{r,s}$.  From
\[
    \omega=\frac{1+\mathbf i+\mathbf j+\mathbf k}{2},
    \qquad
    \tau=\frac{1+\mathbf i}{\sqrt2},
\]
we have
\[
\begin{aligned}
    \omega^0&=1,\\
    \omega&=\frac{1+\mathbf i+\mathbf j+\mathbf k}{2},\\
    \omega^2&=\frac{-1+\mathbf i+\mathbf j+\mathbf k}{2},
\end{aligned}
\qquad
\begin{aligned}
    \ol{\omega}^{\,0}&=1,\\
    \ol{\omega}&=\frac{1-\mathbf i-\mathbf j-\mathbf k}{2},\\
    \ol{\omega}^{\,2}&=\frac{-1-\mathbf i-\mathbf j-\mathbf k}{2},
\end{aligned}
\]
and
\[
    \ol{\tau}=\frac{1-\mathbf i}{\sqrt2}.
\]
Thus
\[
    \sqrt2\,A_{r,s}
    =
    \ol{\omega}^{\,s}(1-\mathbf i)\omega^r.
\]
The nine products are
\[
\begin{array}{c|ccc}
 & s=0 & s=1 & s=2\\ \hline
r=0
&
1-\mathbf i
&
-\mathbf i-\mathbf k
&
-1-\mathbf k
\\[1mm]
r=1
&
1+\mathbf j
&
1-\mathbf k
&
-\mathbf j-\mathbf k
\\[1mm]
r=2
&
\mathbf i+\mathbf j
&
1+\mathbf i
&
1-\mathbf j
\end{array}
\]
where every entry in the table is $\sqrt2\,A_{r,s}$.

It remains to read off the real parts of $q_zA_{r,s}$.  Since
$q_{(0,0)}=1$, $q_{(1,0)}=\mathbf i$, $q_{(0,1)}=\mathbf j$, $q_{(1,1)}=\mathbf k$, the previous table immediately gives all four possibilities for $z$
for each pair $(r,s)$:
\[
\begin{array}{c|rrrr|c}
(r,s)
&
z=(0,0)
&
z=(1,0)
&
z=(0,1)
&
z=(1,1)
&
\{z:\RePart(q_zA_{r,s})\ne0\}
\\ \hline
(0,0)&  1&  1&  0&  0&
\{(0,0),(1,0)\}\\
(0,1)&  0&  1&  0&  1&
\{(1,0),(1,1)\}\\
(0,2)& -1&  0&  0&  1&
\{(0,0),(1,1)\}\\
(1,0)&  1&  0& -1&  0&
\{(0,0),(0,1)\}\\
(1,1)&  1&  0&  0&  1&
\{(0,0),(1,1)\}\\
(1,2)&  0&  0&  1&  1&
\{(0,1),(1,1)\}\\
(2,0)&  0& -1& -1&  0&
\{(1,0),(0,1)\}\\
(2,1)&  1& -1&  0&  0&
\{(0,0),(1,0)\}\\
(2,2)&  1&  0&  1&  0&
\{(0,0),(0,1)\}
\end{array}
\]
The numerical entries in this table are the values of
$\sqrt2\,\RePart(q_zA_{r,s})$. Recall that only their vanishing is important.

We now compare the last column of the last table with the binary condition in the
statement.  Write $z=(z_1,z_2)$. Recall that
\[
    \delta_0=(0,1),\qquad
    \delta_1=(1,0),\qquad
    \delta_2=(1,1),
\]
and
\[
    \ell_0(z_1,z_2)=z_2,\qquad
    \ell_1(z_1,z_2)=z_1,\qquad
    \ell_2(z_1,z_2)=z_1+z_2.
\]
All additions in the following table are in $\FF_2$:
\[
\begin{array}{c|c|c|c}
(r,s)
&
r+s
&
\ell_{r+s}(z+\delta_r)
&
\{z\mid \ell_{r+s}(z+\delta_r)=1\}
\\ \hline
(0,0)&0&z_2+1&
\{(0,0),(1,0)\}\\
(0,1)&1&z_1&
\{(1,0),(1,1)\}\\
(0,2)&2&z_1+z_2+1&
\{(0,0),(1,1)\}\\
(1,0)&1&z_1+1&
\{(0,0),(0,1)\}\\
(1,1)&2&z_1+z_2+1&
\{(0,0),(1,1)\}\\
(1,2)&0&z_2&
\{(0,1),(1,1)\}\\
(2,0)&2&z_1+z_2&
\{(1,0),(0,1)\}\\
(2,1)&0&z_2+1&
\{(0,0),(1,0)\}\\
(2,2)&1&z_1+1&
\{(0,0),(0,1)\}.
\end{array}
\tag{****}
\]
The last columns of tables agree in every one of the nine
cases. Hence, for $z=u+v$,
\[
    \RePart(q_zA_{r,s})\ne0
    \quad\Longleftrightarrow\quad
    \ell_{r+s}(z+\delta_r)=1.
\]
Substituting $z=u+v$ gives
\[
    \ip{c}{d}\ne0
    \quad\Longleftrightarrow\quad
    \ell_{r+s}(u+v+\delta_r)=1,
\]
which proves the lemma.
\end{proof}

\begin{lemma}
\label{lem:automatic-parity}
Let
\[
    (a,b,c)\in\Tau_{1,r},
    \qquad
    (a',b',c')\in\Tau_{2,s}.
\]
Then among
\[
    \ip{a}{a'},
    \qquad
    \ip{b}{b'},
    \qquad
    \ip{c}{c'},
\]
either all zeros or exactly one is zero.
\end{lemma}

\begin{proof}
Let
\[
    (u,v,w)\in\BB_0,
    \qquad
    (u',v',w')\in\BB_{\delta_s}
\]
be such that $a\in \tau\omega^r(\mathcal{L}(u))$, $b\in \omega^r(\mathcal{L}(v))$, $c\in \omega^r(\mathcal{L}(w))$ and $a'\in \omega^s(\mathcal{L}(u'))$, $b'\in \tau\omega^s(\mathcal{L}(v'))$, $c'\in \tau\omega^s(\mathcal{L}(w'))$. Thus
\[
    u+v+w=0,
    \qquad
    u'+v'+w'=\delta_s.
\]
We have $a\in \DD_r$ and 
$a'\in\CC_s$. By Lemma~\ref{lem:cross-incidence}, its incidence bit is
\[
    \ell_{r+s}(u+u'+\delta_s).
\]
The second and third components compare elements of $\CC_r$ with elements of
$\DD_s$, and their incidence bits are
\[
    \ell_{r+s}(v+v'+\delta_r),
    \qquad
    \ell_{r+s}(w+w'+\delta_r).
\]
The parity of the number of nonzero component scalar products is therefore
\[
\begin{aligned}
    &\ell_{r+s}(u+u'+\delta_s)
    +\ell_{r+s}(v+v'+\delta_r)
    +\ell_{r+s}(w+w'+\delta_r)
    \\
    &\quad=
    \ell_{r+s}\!\left(
       (u+v+w)+(u'+v'+w')+\delta_s
    \right)
  =
    \ell_{r+s}(0+\delta_s+\delta_s)
    =0
\end{aligned}
\]
in $\FF_2$. Since there are
only three components, their number of nonzero scalar products is either
$0$ or $2$.
\end{proof}

\subsection{Final steps of the proof}

\begin{lemma}
\label{lem:basic-properties}
The set $\XX_0$ consists of $840$ unit vectors and is antipodal.
\end{lemma}

\begin{proof}
The three axial families have total size $72$. By
Lemma~\ref{prop:tau-cardinality}, each mixed family has size $384$.
By construction, the five families are disjoint. Hence
$|\XX_0|=72+384+384=840$. 
Their elements have norm one.
Every lift $\widehat{\BB}_\eta$ contains all possible sign choices, so
$\XX_0=-\XX_0$.
\end{proof}

\begin{lemma}
\label{lem:axial-pairs}
If $x,y\in\XX_0$ are distinct and at least one belongs to
$\EE_1\cup\EE_2\cup\EE_3$, then
\[
    \ip{x}{y}\le\frac12.
\]
\end{lemma}

\begin{proof} Recall that $24$-cell is the only possible (up to isometry) kissing arrangement of size $24$ in ${\mathbb R}^4$.
Thus, within one of $\EE_s$, the assertion is the usual scalar-product bound for
the regular $24$-cell. By construction, different $\EE_s$ and $\EE_{s'}$ are orthogonal.

Let us now look at the interation between vectors of $\EE_1\cup\EE_2\cup\EE_3$ and the mixed families. 
For example, let $x=(c_0,0,0)\in\EE_1$. Its scalar product with a point of
$\MM_1$ is
\[
    \frac1{\sqrt2}\ip{c_0}{a}
    \in\left\{-\frac12,0,\frac12\right\},
\]
by part (iii) of Lemma~\ref{lem:coset-scalar-products}. Its scalar product with a
point of $\MM_2$ is
\[
    \frac12\ip{c_0}{a'}\le\frac12.
\]
The argument for $\EE_2$ is symmetric. For $\EE_3$, the two corresponding upper
bounds are $\frac12$ and $\frac{1}{2\sqrt2}$.
\end{proof}
Now we may concentrate completely on mixed families.
\begin{lemma}
\label{lem:within-mixed}
Distinct points of $\MM_1$ have scalar product at most $\frac12$. The same is
true for $\MM_2$.
\end{lemma}

\begin{proof}
Consider two points arising from the same $\Tau_{1,r}$ or $\Tau_{2,r}$. Elements of either set are triples obtained by applying componentwise quaternionic rotations to sign-lifted elements of $\mathcal K$, together with an additional weighting of the three components.

Before weighting, each pair of corresponding components belongs to a rotated copy $q\QQ$, with $|q|=1$. Hence their scalar product lies in $\{-1,0,1\}$,
and is nonzero exactly when the two components coincide up to sign. In that case, the scalar product is $1$ or $-1$, according as the corresponding sign lifts are equal or antipodal.

Now every affine subspace
$\BB_\eta=\{(u,v,w)\mid u+v+w=\eta\}$
has the following property: for any $a,b\in\BB_\eta$, equality in any two components forces equality in the third. Consequently, two distinct elements $a,b\in\BB_\eta$ can agree in at most one component.

The three weights entering the total scalar product are
\[
\frac12,\qquad \frac14,\qquad \frac14,
\]
possibly in a different order. Therefore, for two distinct sign-lifted elements $a,b\in\BB_\eta$, after the corresponding quaternionic rotations and weighting, at most one component can contribute positively to the scalar product. Since the largest possible weight is $\frac12$, their total scalar product is at most
$\frac12$.

If $a=b$ but the sign lifts are distinct, at least
one component sign changes. Changing a sign in a component of weight $q$
reduces the scalar product from $1$ by $2q$. The smallest weight is $\frac14$,
so the largest value between distinct sign lifts is $1-2\cdot\frac14=\frac12$.

Finally, if the two triples belong to different $\Tau_{1,r}$, $\Tau_{1,s}$ (or,  $\Tau_{2,r}$, $\Tau_{2,s}$), $r\ne s$, then in all
three components they belong to distinct quaternionic cosets. By part (ii) of
Lemma~\ref{lem:coset-scalar-products}, each component scalar product is
at most $\frac12$, and therefore
\[
    \ip{x}{y}
    \le
    \frac12\cdot\frac12
    +\frac14\cdot\frac12
    +\frac14\cdot\frac12
    =\frac12.
\]
\end{proof}
Interaction between the two mixed families is handled by the following lemma.
\begin{lemma}
\label{lem:cross-mixed}
If $x\in\MM_1$ and $y\in\MM_2$, then
\[
    \ip{x}{y}\le\frac12.
\]
\end{lemma}

\begin{proof}
By construction, we have
\[
    x=\left(\frac a{\sqrt2},\frac b2,\frac c2\right),
    \qquad
    y=\left(\frac{a'}2,\frac{b'}{\sqrt2},\frac{c'}2\right).
\]
Then
\[
    \ip{x}{y}
    =
    \frac1{2\sqrt2}\ip{a}{a'}
    +\frac1{2\sqrt2}\ip{b}{b'}
    +\frac14\ip{c}{c'}.
\]
Every component scalar product is zero or has absolute value $1/\sqrt2$.
By Lemma~\ref{lem:automatic-parity}, either all three vanish or exactly two
are nonzero. If the first two are nonzero, then one of corresponding components should be from  $\mathcal C$ and another from $\mathcal D$, and by part (iii) of
Lemma~\ref{lem:coset-scalar-products}, we have
\[
    \ip{x}{y}\le \frac1{2\sqrt2}\frac1{\sqrt2}
    +\frac1{2\sqrt2} \frac1{\sqrt2}=\frac14+\frac14=\frac12.
\]
If the third and exactly one of the first two are nonzero, then
\[
    \ip{x}{y}
    \le 
    \frac14+\frac1{4\sqrt2}
    =\frac{2+\sqrt2}{8}
    <\frac12.
\]
\end{proof}

\begin{proof}[Proof of Theorem~\ref{thm:canonical-kissing}]
The cardinality, unit norms, and antipodality follow from
Lemma~\ref{lem:basic-properties}. Every pair of distinct points either
involves an axial point, lies within one mixed family, or has one point in
each mixed family. The three cases are covered by
Lemmas~\ref{lem:axial-pairs}, \ref{lem:within-mixed}, and
\ref{lem:cross-mixed}, respectively.
\end{proof}

\section*{Distinction from previous $840$-point arrangements}
\label{sec:distinction-previous-840}

The quaternionic construction introduced in Section~\ref{sec:canonical-arrangement} is not isometric to any of the previously known
$840$-point kissing arrangements in $\RR^{12}$.  There are two broad families of such arrangements from which we need to distinguish it. We first separate it from all $1579$ arrangements obtained from
maximum independent sets of the signed Johnson's graph $J_{\pm}(12,4)$~\cite{takhanov2026classificationindependentsetssigned}.  Second we separate it from the
continuous family of Clebsch-based arrangements obtained by varying the two
$48$-systems in the construction of~\cite{takhanov2026structurekissingarrangementsmathbb}.

\subsection{Distinction from $1579$ arrangements from~\cite{takhanov2026classificationindependentsetssigned}}

For a finite spherical configuration, by its \emph{inner-product profile} we
mean the multiplicity of each inner product among unordered pairs of distinct
points.  A direct exact computation from the quaternionic construction gives
the following profile.

\begin{center}
\begin{tabular}{c r}
\toprule
inner product & number of unordered pairs \\
\midrule
$-1$ & $420$ \\
$-\frac12$ & $32928$ \\
$-\frac{2+\sqrt2}{8}$ & $18432$ \\
$-\frac{\sqrt2}{4}$ & $2304$ \\
$-\frac14$ & $43008$ \\
$-\frac{2-\sqrt2}{8}$ & $18432$ \\
$0$ & $121752$ \\
$\frac{2-\sqrt2}{8}$ & $18432$ \\
$\frac14$ & $43008$ \\
$\frac{\sqrt2}{4}$ & $2304$ \\
$\frac{2+\sqrt2}{8}$ & $18432$ \\
$\frac12$ & $32928$ \\
\bottomrule
\end{tabular}
\end{center}
In fact, the latter table is not needed for our argument --- it only matters that the profile contains the irrational inner
product $\frac{2+\sqrt2}{8}$. The latter can be seen from the proof of Lemma~\ref{lem:cross-mixed}.

This immediately distinguishes the quaternionic arrangement from all the
$1579$ arrangements classified in~\cite{takhanov2026classificationindependentsetssigned}.  Indeed, those
arrangements arise from maximum independent sets in $J_{\pm}(12,4)$.  Before
normalization, every vertex of $J_{\pm}(12,4)$ has exactly four nonzero
coordinates, each equal to $\pm1$, and the corresponding $840$-point kissing
arrangement is obtained by appending the $24$ vectors $\pm2e_i$.  After
normalizing, every inner product between two distinct points belongs to
\[
\left\{-1,-\frac34,-\frac12,-\frac14,0,
             \frac14,\frac12\right\}.
\]
Thus, none of the $1579$ arrangements can have the profile containing $\frac{2+\sqrt2}{8}$.

\subsection{Distinction from Clebsch-based arrangements}

We next compare the quaternionic construction with the Clebsch-based family
from~\cite{takhanov2026structurekissingarrangementsmathbb}.  We recall only the part of that construction
needed here.  For the present purpose, a \emph{$48$-system}
is a set of $48$ unit vectors $v\in\mathbb R^6$ such that distinct vectors have
inner product at most $\frac12$ and
\[
|v_r|\le \frac12\qquad (r=1,\ldots,6).
\]
These are the only properties of the $48$-systems that will be used here.  In
particular, adjoining the twelve signed coordinate vectors
$\pm e_1,\ldots,\pm e_6$ gives a $60$-point kissing configuration in
$\mathbb R^6$. 

Take two such  $60$-point kissing configurations and place them in the two
orthogonal factors of
$\mathbb R^{12}=\mathbb R^6\oplus\mathbb R^6$. 
Then add the fixed set of $720$ bridge vectors, i.e. vectors of the form
\[
\frac12\bigl(\varepsilon_i e_i+\varepsilon_j e_j+
\varepsilon_k e_{6+k}+\varepsilon_\ell e_{6+\ell}\bigr),
\qquad
\varepsilon_i,\varepsilon_j,\varepsilon_k,\varepsilon_\ell\in\{\pm1\},
\]
where $\{i,j\}$ and $\{k,\ell\}$ are edges of the
same color in the fixed $1$-factorization of the full graph on $\{1,...,6\}$.
The following argument shows that no choice of the two $48$-systems can yield
the quaternionic arrangement.

\begin{theorem}
No Clebsch-based $840$-point arrangement obtained from two $48$-systems and
the $720$ bridge vectors above is isometric to the quaternionic $840$-point
arrangement.
\end{theorem}

\begin{proof}
For a configuration $X\subset\mathbb R^{12}$, consider the frame operator
\[
S_X=\sum_{x\in X}xx^{\top}.
\]
Its eigenvalues are invariant under orthogonal transformations and relabeling
of the points.

Let $X$ be any of Clebsch-based arrangements. We first compute the contribution of the $720$ bridge vectors to $S_X$.  Fix one of the
$12$ coordinates.  It occurs in exactly
\[
5\cdot3\cdot16=240
\]
bridge vectors: there are five edges of $K_6$ containing the corresponding
vertex, three edges in the matching of the same color in the other factor, and
$16$ sign choices.  Since the square of every nonzero bridge coordinate is
$\frac14$, the corresponding diagonal entry of the bridge frame operator is
\[
240\cdot\frac14=60.
\]
All off-diagonal entries vanish because the $16$ sign choices occur
symmetrically.  Hence the $720$ bridge vectors contribute $60I_{12}$ to $S_X$.

The two $60$-point blocks contribute positive semidefinite operators supported
on the first and second copies of $\mathbb R^6$, respectively.  Therefore the
frame operator of every Clebsch-based arrangement has the form
\[
60I_{12}+
\begin{pmatrix}
A&0\\
0&B
\end{pmatrix},
\]
where $A$ and $B$ are positive semidefinite $6\times6$ matrices.  In
particular,
\[
\lambda_{\min}(S_X)\ge 60.
\]

For the quaternionic arrangement, the frame operator can be read directly
from the construction.  Each of the three axial regular $24$-cells contributes
$6I_4$ on its quaternionic coordinate space.  In each $384$-point mixed
family, every quaternionic coordinate is uniformly distributed over a regular
$24$-cell: every point of that $24$-cell occurs $16$ times.  The independent
sign symmetry in the ternary relations makes the off-diagonal quaternionic
blocks vanish.  Since the squared coordinate weights are
\[
\left(\frac12,\frac14,\frac14\right)
\quad\text{and}\quad
\left(\frac14,\frac12,\frac14\right)
\]
for the two mixed families, their frame operators are respectively
\[
\operatorname{diag}(48I_4,24I_4,24I_4)
\quad\text{and}\quad
\operatorname{diag}(24I_4,48I_4,24I_4).
\]
After adding the three axial $24$-cells, the frame operator of the quaternionic
$840$-point arrangement is therefore
\[
\operatorname{diag}(78I_4,78I_4,54I_4).
\]
Thus, its smallest eigenvalue is $54$.  This is impossible
for a Clebsch-based arrangement, whose smallest eigenvalue is at least
$60$.  Hence the two configurations cannot be isometric.
\end{proof}

\end{document}